\documentclass[12pt, reqno]{amsart}
\usepackage{palatino}
\usepackage{amsmath,amssymb,amsxtra,color,calligra,mathrsfs,tcolorbox}
\usepackage[all,cmtip]{xy}
\usepackage{xcolor}
\colorlet{mdtRed}{red!50!black}
\definecolor{dblue}{rgb}{0,0,.6}
\usepackage[colorlinks]{hyperref}
\hypersetup{linkcolor=mdtRed,citecolor=dblue,filecolor=dullmagenta,urlcolor=mdtRed}
\usepackage[all]{xy}
\usepackage{tikz,tikz-cd,tkz-graph,enumerate}

\usetikzlibrary{matrix,arrows,decorations.pathmorphing}

\DeclareMathOperator{\Malpha}{\mathcal{M}_{\alpha}} 
\DeclareMathOperator{\Mbeta}{\mathcal{M}_{\beta}}
\DeclareMathOperator{\Mgamma}{\mathcal{M}_{\gamma}}
\DeclareMathOperator{\phialpha}{\phi_{\alpha}}
\DeclareMathOperator{\phibeta}{\phi_{\beta}}
\DeclareMathOperator{\psialpha}{\psi_{\alpha}}
\DeclareMathOperator{\psibeta}{\psi_{\beta}}
\DeclareMathOperator{\Sigmagamma}{\Sigma_{\gamma}}
\DeclareMathOperator{\Pnalpha}{\mathbb{P}^{n_{\alpha}}} 
\DeclareMathOperator{\Pnbeta}{\mathbb{P}^{n_{\beta}}}
\DeclareMathOperator{\Q}{\mathbb{Q}}
\DeclareMathOperator{\xtl}{[\{x\}\times \textit{l}]}
\DeclareMathOperator{\CH}{\textnormal{CH}}

\newtheorem{theorem}{Theorem}[section]
\newtheorem{lemma}[theorem]{Lemma} 
\newtheorem{proposition}[theorem]{Proposition}

\theoremstyle{definition}
\newtheorem{definition}[theorem]{Definition}
\newtheorem{remark}[theorem]{Remark}

\numberwithin{equation}{section} 

\begin{document}
	
\baselineskip=15.5pt 
	
\title[Chow group of 1-cycles of moduli of parabolic bundles]{Chow group of 1-cycles of the moduli of parabolic bundles over a curve}

\author{Sujoy Chakraborty} 
\address{School of Mathematics, Tata Institute of Fundamental Research, Homi Bhabha Road, Colaba, Mumbai 400005, India.}
\email{sujoy@math.tifr.res.in}
\thanks{E-mail : sujoy@math.tifr.res.in}
\thanks{Address : School of Mathematics, Tata Institute of Fundamental Research, Homi Bhabha Road, Colaba, Mumbai 400005, India.}

\begin{abstract}
	We study the Chow group of 1-cycles of the moduli space of semistable parabolic vector bundles of fixed rank, determinant and a generic weight over a nonsingular projective curve over $\mathbb{C}$ of genus at least 3. We show that, the Chow group of 1-cycles remains isomorphic as we vary the generic weight. As a consequence, we can give an explicit description of the Chow group in the case of rank 2 and determinant $\mathcal{O}(x)$, where $x\in X$ is a fixed point, which extends the earlier result of \cite[Main Theorem]{CH}.
\end{abstract}

\maketitle

\section{Introduction}
	
Let $X$ be a nonsingular projective curve of genus $g\geq 3$ over $\mathbb{C}$. Let $\mathcal{M}(r,\mathcal{L})$ denote the moduli space of isomorphism classes of stable vector bundles of rank r and fixed determinant $\mathcal{L}$ over $X$. Let us moreover fix a set of $n$ distinct closed points $S$ over $X$, referred to as \textit{parabolic points}, and let $\mathcal{M}(r,\mathcal{L}, \alpha)$ denote the moduli space of isomorphism classes of parabolic stable vector bundles of rank r, determinant $\mathcal{L}$, \textit{full} flags along the parabolic points, and generic weight $\alpha$ over $X$. The Chow groups of these moduli spaces are interesting objects to study. Not much is known about the explicit desctiption of these Chow groups (see, e.g. the Introduction in \cite{CH} for a list of known results in case of moduli of vector bundles). Here our aim is to study the Chow group of 1-cycles for the moduli space $\mathcal{M}(r,\mathcal{L}, \alpha)$. We outline the content of the paper and our strategy of proof below:

In section 2, we briefly recall the notions necessary for our discussions, like (parabolic) semistability and stability of (parabolic) vector bundles, their moduli spaces, Chow groups and so on. In section 3, we begin our study of the Chow group of 1-cycles for the moduli of Parabolic bundles over $X$. The extra data of parabolic structure makes it necessary to study the effect of varying the weights. Below, we denote $\mathcal{M}(r, \mathcal{L}, \alpha)$ by just $\Malpha$, since the r and $\mathcal{L}$ will remain fixed. The main result of section 3 is the following:

\begin{theorem} [Theorem \ref{thm-3.10}]
	For any two generic weights $\alpha$ and $\beta$, there exists a canonical isomorphism 
	\[\CH_1^{\Q}(\Malpha) \cong \CH_1^{\Q}(\Mbeta).\]
\end{theorem}

Our strategy for proving this is to first prove the result for generic weights $\alpha$ and $\beta$ which are separated by a single hyperplane (called \textit{walls}, cf. Section 2) inside the set of all possible weights (denoted by $V_m$, cf. Section 2), and choosing the weight $\gamma$ which is the point of intersection of the hyperplane and the line joining $\alpha$ and $\beta$. We note that $\gamma$ is not a generic weight, and hence $\Mgamma$ is a normal projective variety. To relate the Chow groups for $\Malpha$ and $\Mbeta$, our approach is to use \cite[Theorem 4.1]{BY} (equivalently \cite[Theorem 3.1]{BH}), which says that there exist maps

\[
\xymatrix{
	\Malpha \ar[rd]_{\phialpha}
	& 
	&\Mbeta \ar[ld]^{\phibeta} \\
	&\Mgamma
}
\] 

which act as resolution of singularities for $\Mgamma$. Next, we consider the fibre product $ \mathcal{N} := \Malpha \underset{\Mgamma}{\times}\Mbeta$, which is actually a common blow-up of $\Malpha$ and $\Mbeta$ along suitable subvarieties(cf. discussion in the Introduction in \cite{BH} above Section 2). We use the Blow-up formula for Chow groups as in \cite[Theorem 9.27]{Voi} to first relate both $\CH_{1}^{\Q}(\Malpha)$ and $\CH_{1}^{\Q}(\Mbeta)$ with $\CH_{1}^{\Q}(\mathcal{N})$, and then through a series of manipulations, we finally get our required isomorphism. To conclude for all generic weights is straightforward from here.

As a consequence, we can give an explicit description of the Chow group when rank is 2 and determinant is $\mathcal{O}_X(x)$ for some closed point $x\in X$. We do this in Section 4, where we prove the following result:
\begin{theorem}[Theorem \ref{thm-4.4}]
	In case of rank 2 and determinant $\mathcal{O}_X(x)$, for any generic weight $\alpha$, we have
	 \[\CH_1^{\Q}(\mathcal{M}_\alpha) \cong \Q^n \oplus\,\CH_0^{\Q}(X),\,\,\textnormal{where}\,\,n=|S|.\]
\end{theorem}

The main idea here is to show that for a sufficiently small generic weight $\alpha$ (as in Proposition \ref{prop-4.1}), $\Malpha$ has the structure of a $(\mathbb{P}^1)^n$-bundle over $\mathcal{M}$; we show this in Lemma \ref{projbund}. This, together with projective bundle formula for Chow groups as in \cite[Theorem 9.25]{Voi}, enable us to explicitly write down $\CH_{1}^{\Q}(\Malpha)$ in terms of $\CH_{1}^{\Q}(\mathcal{M})$, and the latter is known to be isomorphic to $\CH_0^{\Q}(X)$ by \cite[Main Theorem]{CH}. The case for arbitrary generic weight then follows from Theorem 1.1.

\section{Preliminaries}

\subsection{Semistability and stability of vector bundles}
	
Let $X$ be a nonsingular projective curve over $\mathbb{C}$. Let $E$ be a holomorphic vector bundle of rank $r$ over $X$. \\
Here onwards, by a $variety$ we will always mean an irreducible quasi-projective variety.
	
\begin{definition}[Semistability and stability]
	The \textit{degree} of $E$, denoted $deg(E)$, is defined as the degree of the line bundle $det(E) := \wedge^r E$. The \textit{slope} of $E$, denoted $\mu(E)$, is defined as 
	\[\mu(E) := \dfrac{deg(E)}{r}\]

	$E$ is called \textit{semistable (resp. stable)}, if for any sub-bundle $F\hookrightarrow E, \, 0< rank(F) < r$, we have 
	\[\mu(F) \, \substack{\leq \\ (resp.\,<)}\,\mu(E).\]
\end{definition}
	
It is easy to check that if $gcd(r, deg(E))=1$, then the notion of semistability and stability coincide for a vector bundle $E$.
	
\subsection{Moduli space of vector bundles}\label{sec-2.2}
We briefly recall the notion of the moduli space of vector bundles over $X$. If $E$ is a semistable bundle of rank $r$, then there exists a \textit{Jordan-H\"older filtration} for $E$ given by 
\[E = E_k \supsetneq E_{k-1} \supsetneq \cdots \supsetneq E_1 \supsetneq 0\]

The filtration is not unique, but the associated graded object $\textrm{gr}(E) := \bigoplus_{i=1}^{k} E_i/E_{i-1}$ is unique upto isomorphism. Two vector bundles $E$ and $E'$ are called \textit{S-equivalent} if $\textrm{gr}(E)\cong \textrm{gr}(E')$. When $E, E'$ are stable, being S-equivalent is same as being isomorphic as vector bundles over $X$.

The moduli space of S-equivalence classes of vector bundles of rank $r$ and determinant $\mathcal{L}$ on $X$, denoted $\mathcal{M}({r,\mathcal{L}})$, is a normal projective variety of dimension $(r^2-1)(g-1)$; its singular locus is given by the strictly semistable bundles (bundles which are semistable but not stable). 

In the case when $gcd(r,deg(\mathcal{L)})=1$, $\mathcal{M}({r,\mathcal{L}})$ is the isomorphism class of stable vector bundles on $X$. It is a nonsingular projective variety; moreover, it is a fine moduli space.

When $r,\mathcal{L}$ are fixed, we shall denote the moduli space by $ \mathcal{M} $, when there is no scope for confusion.
\subsection{Parabolic bundles and stability}

	\begin{definition}[Parabolic bundles]
		Let us fix a set $S$ of $n$ distinct closed points on $X$. A \textit{parabolic vector bundle of rank r on X} is a holomorphic vector bundle $E$ on $X$ with a \textit{parabolic structure along points of S}. By this, we mean a collection weighted flags of the fibres of $E$ over each point $p\in S$:
		\begin{align}
		E_p &= E_{p,1} \supsetneq E_{p,2} \supsetneq ... \supsetneq E_{p,s_p} \supsetneq E_{p, s_{p+1}}= 0, \\
		0 &\leq \alpha_{p,1} < \alpha_{p,2} < ... \,< \alpha_{p,s_p}\, < 1,
		\end{align}
		
		where $s_p$ is an integer between $1$ and $r$. The real number $\alpha_{p,i}$ is called the \textit{weight attached to the subspace} $E_{p,i}$. 	
		The \textit{multiplicity} of the weight $\alpha_{p,i}$ is the integer $m_{p,i} := dim(E_{p,i}) - dim(E_{p,i-1})$. Thus $\sum_i m_{p,i} = r$. 
		We call the flag to be $full$ if $s_p=r,$ or equivalently $m_{p,i} =1 \,\forall i$.
	\end{definition}
	
	Let $\alpha := \{(\alpha_{p,1}, \alpha_{p,2}, ..., \,\alpha_{p,s_p}\,)\,|\,p\in S\}$ and $m:= \{(m_{p,1},...,\,m_{p,s_p}\,)\,|\, p\in S\}$. We call the tuple $(r,\,\mathcal{L},\,m,\,\alpha)$ as the \textit{parabolic data} for the parabolic bundle $E$, where $\mathcal{L} := det(E)$. 
	Usually we denote the parabolic bundle as $E_*$ to distinguish from the underlying vector bundle $E$.
	
	\begin{definition}[Parabolic degree and slope]
		The \textit{degree} of a parabolic bundle $E_*$ is defined as $deg(E)$, $E$ being the underlying vector bundle. The \textit{Parabolic degree} of $E_*$, denoted $Pardeg(E),$ is defined as
		\[Pardeg(E_*):= deg(E) + \sum_{p\in S}\sum_{i=1}^{s_P}m_{p,i}\alpha_{p,i}.\]
		
		The \textit{parabolic slope} of $E_*$ is defined as
		\[Par\mu(E_*) := \dfrac{Pardeg(E_*)}{rank(E)}.\]
	\end{definition}
	
	\begin{definition}[Parabolic semistability and stability]
		Any vector sub-bundle $F\hookrightarrow E$ obtains a parabolic structure in a canonical way: For each $p\in S$, the flag at $F_p$ is obtained intersecting $F_p$ with the flag at $E_p$, and the weight attached to the subspace $F_{p,j}$ is $\alpha_k$, where $k$ is the largest integer such that $F_{p,j}\subseteq E_{p,k}$. (for more details see \cite[Definition 1.7]{MS}.) We call the resulting parabolic bundle to be a \textit{parabolic sub-bundle,} and denote it by $F_*$.
		
		A parabolic bundle $E_*$ is called \textit{parabolic semistable (resp. parabolic stable)}, if for every proper sub-bundle $F\hookrightarrow E$ we have
		\[Par\mu(F_*)\, \substack{\leq \\ (resp. <)} \,Par\mu(E_*).\]  
	\end{definition}
	
\subsection{Generic weights and walls}

We briefly recall the notion of \textit{generic weights} and \textit{walls}. For more details we refer to \cite{BH,BY}.

Fix a set $S$, rank $r$, line bundle $\mathcal{L}$ on $X$ and multiplicities $m$ as defined above. Let $\Delta^r:= \{(a_1,..., a_r)\, | \, 0\leq a_1 \leq ... \leq a_r <1\}$, and define $W := \{\alpha : S \rightarrow \Delta^r\}$. Note that the elements of $W$ determine both weights and the multiplicities at the parabolic points, and hence a parabolic data. Conversely, given any parabolic data $(r, \mathcal{L}, m, \alpha)$, we can associate a map $S \rightarrow \Delta^r$, by repeating each weight $\alpha_{p,i}$ according to its multiplicty $m_{p,i}$.
This leads to a natural notion of when a given weight $\alpha$ is \textit{compatible} with the multiplicity $m$. 
The set of all weights compatible with $m$ is a product of $|S|$-many simplices. We denote by $V_m$ the set of all weights compatible with $m$.

Let $\alpha \in V_m$. If a parabolic bundle $E_*$ with data $(r, \mathcal{L}, m, \alpha)$ is parabolic semistable but not parabolic stable, then it would contain a parabolic sub-bundle with same parabolic slope. It is easy to see that this gives a linear condition on $V_m$, i.e. such weights belong to the intersection of a hyperplane with $V_m$.

There can be only finitely many such hyperplanes (see \cite{BY,BH}); call them $H_1,...,H_l$.

\begin{definition}(Walls and generic weights)
	We call each of the intersections $H_i \cap V_m$ a \textit{wall} in $V_m$. There are only finitely many such walls.
	
	We call the connected components of $V_m \setminus \cup_{1\leq i \leq l}H_i$ as \textit{chambers}, and weights belonging to these chambers are called \textit{generic}.
	
\end{definition}

Clearly, for weights in $V_m \setminus \cup_{1\leq i \leq l}H_i$, a parabolic bundle is parabolic semistable iff it is parabolic stable.

\subsection{Moduli of parabolic bundles}\label{sub-2.5}
Again, we briefly recall the notion of moduli space of parabolic semistable bundles over $X$. The construction is analogous to section \ref{sec-2.2}; for details we refer to \cite{MS}.\\

for a parabolic semistable bundle $E_*$ with fixed parabolic data $(r,\mathcal{L},m,\alpha)$, there exists a Jordan-Holder filtration, and an associated graded object $gr_{\alpha}(E_*)$ analogous to section \ref{sec-2.2}. Again, we call two parabolic semistable bundles to be $S$-equivalent if their associated graded objects are isomorphic. Let $\mathcal{M}(r,\mathcal{L},m,\alpha)$ denote the moduli space of S-equivalence classes of parabolic semistable bundles over $X$ with parabolic data $(r,\mathcal{L},m,\alpha).$ It is a normal projective variety, with singular locus given by the strictly semistable bundles. When $r,\mathcal{L},m$ are fixed, we will denote the moduli space by $\Malpha$ if no confusion occurs.

For generic weight $\alpha$, $\mathcal{M}_\alpha = $ moduli space of isomorphism classes of parabolic stable bundles on $X$, is a nonsingular projective variety; moreover, it is a fine moduli space (\cite[Proposition 3.2]{BY}). 

 \subsection{Chow groups}
For a variety $Y$ over $\mathbb{C}$,  let $Z_k(Y)$ denote the free abelian group generated by the irreducible $k$-dimensional closed subvarieties of $Y$. The \textit{Chow group of k-cycles}, denoted $\CH_k(Y)$, is given by 
\[\CH_k(Y) := \dfrac{Z_k(Y)}{\sim},\]

where $\sim$ denotes "rational equivalence". We refer to \cite[Section 9]{Voi} and \cite{Ful} for the details regarding Chow groups and the related notions (proper pushforward and flat pullback of cycles, intersection product, Chern class of vector bundles etc.)\\
Let $\CH_k^{\Q}(Y) := \CH_k(Y) \otimes_{\mathbb{Z}}\Q$; this is a $\Q$-vector space. By a slight abuse of notation, throughout the rest of the discussion, we will address $\CH_k^{\Q}(Y)$ as 'Chow group' as well, since no confusion will arise. \\
We recall a few results from \cite{Ful} which we will require in section 3:

\begin{theorem}[\text{\cite[Theoerm 3.3]{Ful}}]\label{thm-2.7}
	Let $E$ be a vector bundle of rank $r$ on $Y$, with pfojection $\pi: E\rightarrow Y$. The flat pull-back 
	\[\pi^* : \CH_{k-r}(Y) \rightarrow \CH_k(E)\]
	is an isomorphism for all $k$. 
\end{theorem}

\begin{definition}[\text{\cite[Definition 3.3]{Ful}}]\label{def-2.8}
	Let $s$ denote the zero section of the bundle $E$ above. Hence $\pi \circ s = \textnormal{Id}_Y$. Then there exist \textit{Gysin homomorphisms}:
	\begin{align*}
	s^*:\CH_k(E) &\rightarrow \CH_{k-r}(Y), \\
	s^*(W) &:= (\pi^*)^{-1}(W)
	\end{align*}
	where $r=rank(E)$.
\end{definition}

(We make a small remark that $s^* \neq \pi_*$).

\begin{lemma}[\text{\cite[Example 3.3.2]{Ful}}]\label{lem-2.8}
	If $s$ is the zero section of a vector bundle $E$ of rank $r$ on $Y$, then
	\[s^*s_*(Z) = c_r(E)\cap Z \,\, \textnormal{for all}\,\, Z\in \CH_*(Y).\]
\end{lemma}

\begin{proposition}[\text{\cite[Proposition 6.7(a)]{Ful}}]\label{prop-2.10}
	Let $Y$ be a nonsingular variety, and $X \overset{i}{\hookrightarrow} Y$ be a nonsingular closed subvariety of codimension $d$, with normal bundle $N$. Let $\widetilde{Y}$ be the blow-up of $Y$ along $X$, and $\widetilde{X}$ be its exceptional divisor. We have a fibre square:
	
	\begin{align*}
		\xymatrix{\widetilde{X} \ar[d]_{g} \ar@{^{(}->} [r]^{j} & \widetilde{Y} \ar [d]^{f} \\
			X \ar@{^{(}->} [r]^{i} & Y
		} 
	\end{align*}
	Let $E:= g^*(N)/\mathcal{O}_N (-1)$ be the Excess normal bundle. Then for all $Z\in \CH_k(X),$
	\[f^*i_*(Z)=j_*(c_{d-1}(E)\cap g^*(Z)).\]
\end{proposition}	
\section{Relation between chow groups of 1-cycles of moduli of parabolic bundles for arbitrary generic weights}
	
	
Fix a set $S$ of parabolic points, rank $r$ and determinant $\mathcal{L}$. We assume that we are working with $full$ flags, i.e. $m_{p,i} =1 \forall p,i$. Consider $V_m$, the set of weights compatible with $m$, as in section 2.4. Recall that $V_m$ is cut out by finitely many walls. Moreover, as the flags are full, $V_m$ contains a generic weight by \cite[Proposition 3.2]{BY}. \\
Let $ \alpha ,\beta \in V_m$ be two generic weights in adjacent chambers separated by a single wall. Let $ H $ be the hyperplane separating $ \alpha $ and $ \beta $. Let $ \gamma $ be the weight lying on $ H $ and the line joining $ \alpha $ and $ \beta $. Then $\Malpha$ and $ \mathcal{M}_{\beta} $ are nonsingular projective varieties, while $ \Mgamma $ is normal projective variety, with the singular locus $ \Sigmagamma \subset \Mgamma$ given by the class of strictly semistable bundles. Note that since $ \gamma $ lies on only one hyperplane in $W$, $ \Sigmagamma $ is nonsingular (\cite[Section 3.1]{BH}).\newline
	
Let us recall the following theorem:

\begin{theorem}[\text{\cite[Theorem 3.1]{BY}}]
	There are canonical projective morphisms
	\[
	\xymatrix{
		\Malpha \ar[rd]_{\phialpha}
		& 
		&\Mbeta \ar[ld]^{\phibeta} \\
		&\Mgamma
	}
	\] 
	
	so that: a) $ \phialpha $ and $ \phibeta $ are isomorphisms along $ \Mgamma \setminus \Sigmagamma $,\newline
	\hspace*{10ex} b) along $ \Sigmagamma $, $ \phialpha $ and $ \phibeta $ are $ \Pnalpha $ and $ \Pnbeta $-bundles respectively,\newline
	\hspace*{6.5ex}and  c) $ codim \Sigmagamma = 1+ n_{\alpha} + n_{\beta} $.
	
\end{theorem}

Since $\Sigmagamma$ is nonsingular and $\phialpha^{-1}(\Sigmagamma), \phibeta^{-1}(\Sigmagamma)$ are projective bundles, they are nonsingular closed subvarieties of $\Malpha,\Mbeta$ respectively.
	
Let $ \mathcal{N} := \Malpha \underset{\Mgamma}{\times}\Mbeta$.
Let $ \psialpha $ and $ \psibeta $ denote the natural maps from $ \mathcal{N}$ to $ \Malpha $ and $ \Mbeta $ respectively. Then according to the discussion in the end of section 1 in \cite{BH},
$\mathcal{N}$ is the common blow-up with exceptional divisor a $(\Pnalpha \times \Pnbeta)$-bundle over $ \Sigmagamma $.
	
Call the exceptional divisor $E$, with $ j: E\hookrightarrow \mathcal{N} $ the inclusion.
	
We have the following diagram:

\[
\xymatrixcolsep{4pc}\xymatrix{
	& E\ar@{^{(}->}[r]^{j} \ar[ld] \ar[rd]	&\mathcal{N} \ar[ld]^(.2){\psialpha} \ar[rd]^{\psibeta} \\
	\phialpha^{-1}(\Sigmagamma) \ar@{^{(}->}[r] \ar[rd]
	& \Malpha \ar[rd]^(.75){\phialpha}
	&\phibeta^{-1}(\Sigmagamma) \ar@{^{(}->}[r] \ar[ld]
	&\Mbeta \ar[ld]^{\phibeta}\\
	& \Sigmagamma \ar@{^{(}->}[r]
	& \Mgamma
}
\]
	
Here $E$ is a $\Pnbeta$-bundle over $\phibeta^{-1}(\Sigmagamma)$ via $\psialpha|_E$, and a $\Pnalpha$-bundle over $\phibeta^{-1}(\Sigmagamma)$ via $\psibeta|_E$.
	
\begin{remark}\label{rem-1}
	From the diagram above, we note that $E\cong \phialpha^{-1}(\Sigmagamma)\underset{\Sigmagamma}{\times} \phibeta^{-1}(\Sigmagamma)$, since
		
	\begin{align*}
	E = \mathcal{N}\underset{\Mbeta}{\times} \phibeta^{-1}(\Sigmagamma) &= (\Malpha\underset{\Mgamma}{\times}\Mbeta)\underset{\Mbeta}{\times}\phibeta^{-1}(\Sigmagamma) \\
	&\cong \Malpha\underset{\Mgamma}{\times}\phibeta^{-1}(\Sigmagamma) \\
	&\cong \phialpha^{-1}(\Sigmagamma)\underset{\Sigmagamma}{\times}\phibeta^{-1}(\Sigmagamma). \hspace{15ex} [\because \phibeta^{-1}(\Sigmagamma) \,\,\textnormal{maps\,\,to}\,\, \Sigmagamma]
\end{align*}
		
\end{remark}

\vspace{5ex}
	
\begin{lemma}\label{lem-3.3}
	$\phialpha^{-1}(\Sigmagamma)$ and $\phibeta^{-1}(\Sigmagamma)$ are rational varieties (i.e. birational to $\mathbb{P}^n$ for some n); hence $\CH_0^{\mathbb{Q}}(\phialpha^{-1}(\Sigmagamma)) \cong \mathbb{Q}\cong \CH_0^{\mathbb{Q}}(\phibeta^{-1}(\Sigmagamma))$.
\end{lemma}
	
\begin{proof}	
	By equation (5) in \cite{BH}, $ \Sigmagamma $ is the product of two smaller dimensional moduli, which are rational (by \cite[Theorem 6.1]{BY}), so $ \Sigmagamma $ is itself rational.
		
	Since $ \phialpha^{-1}(\Sigmagamma) $ and $ \phibeta^{-1}(\Sigmagamma) $ are projective bundles over $ \Sigmagamma $, they are also rational. This proves the first assertion.
		
	Moreover, by \cite[Example 16.1.11]{Ful}, the Chow groups of 0-cycles is a birational invariant; and $\CH_0(\mathbb{P}^n) \cong \mathbb{Z} \,\forall\,n$, so we get the second assertion as well.
\end{proof}
	
Recall the fibre diagram from Remark \ref{rem-1}: \newline
\begin{align*}
\xymatrix{E \ar[d]^{\psialpha|
	_E} \ar[r]_{\psibeta |_E}
	&\phibeta^{-1}(\Sigmagamma) \ar[d]_{\phibeta}^{\substack{\mathbb{P}^{n_\beta}-\\bundle}} \\
	\phialpha^{-1}(\Sigmagamma) \ar[r]^(.6){\phialpha}_(.6){\substack{\mathbb{P}^{n_\alpha}- \\ bundle}} 
	&\Sigmagamma
	}
\end{align*}
\newline
Therefore, if we choose a point $p\in \Sigmagamma$, then by base changing to $\{p\}$, the diagram above transforms to \newline
	
\begin{align*}
\xymatrix{{\Pnalpha \times \Pnbeta} \ar[d]_{p_1} \ar[r]^{p_2} 
	&{ \Pnbeta \cong \phibeta^{-1}(p)} \ar[d] \\
	{\Pnalpha \cong \phialpha^{-1}(p) } \ar[r] &\{p\}
	}
\end{align*}
	
where $p_1,p_2$ denote the first and second projections respectively.
	
Let $ \phialpha^{-1}(p)\cong \Pnalpha \overset{i_\alpha}{\hookrightarrow} \phialpha^{-1}(\Sigmagamma), \psialpha^{-1}(\phialpha^{-1}(p)) \cong \Pnalpha\times\Pnbeta\ \overset{\widetilde{\imath}}{\hookrightarrow} E$ denote the inclusions. We have the fibre diagram \newline
	
\begin{align*}
\xymatrixcolsep{2pc}\xymatrix{{\Pnalpha \times \Pnbeta} \ar[r]^{\cong} \ar[d]_{p_1} 
	& {\psialpha^{-1}(\phialpha^{-1}(p))} \ar@{^{(}->}[r]^(.6){\widetilde{\imath}} \ar[d]
	& E \ar[d]^{\psialpha|_E} \\
	{\Pnalpha} \ar[r]^{\cong}
	& {\phialpha^{-1}(p)} \ar@{^{(}->}[r]^{\imath_\alpha}
	& {\phialpha^{-1}(\Sigmagamma)}
}
\end{align*}
	
Choose a point $x\in \Pnalpha \cong \phialpha^{-1}(p)$. In the following, under slight abuse of notation, we will think of the element $[x]\in \CH_0^{\Q}(\Pnalpha)$ as an element of  $ \CH_0^{\Q}(\phialpha^{-1}(p))$, and we will think of the element $[\{x\} \times \Pnbeta] \in \CH_{n_{\beta}}^{\Q}(\Pnalpha \times \Pnbeta)$ as an element of $ \CH_{n_{\beta}}^{\Q}(\psialpha^{-1}(\phialpha^{-1}(p)))$.\newline
	
\begin{lemma}\label{lem-3.4}
	$(\psialpha|_E)^*((\imath_{\alpha})_*[x]) = \widetilde{\imath}_*[\{x\}\times \Pnbeta]$
\end{lemma}
	
\begin{proof}
	This follows from \cite[Proposition 1.7]{Ful}, since $\psialpha|_E$ is flat, being a projective bundle map, and $\imath_\alpha$ is proper, being a closed immersion.
\end{proof}
	
Now, since $\mathcal{N}$ is the blow-up over $\Malpha$ along $\phialpha^{-1}(\Sigmagamma)$, hence  by \cite[Theorem 9.27]{Voi} there is an isomorphism of Chow groups:
	
\begin{align}
\CH_0^{\Q}(\phialpha^{-1}(\Sigmagamma)) &\oplus \CH_1^{\Q}(\Malpha) \xrightarrow[\sim]{g_{\alpha}} \CH_1^{\Q}(\mathcal{N}) \label{3.1}\\
	\textnormal{given\,\,by}\quad\quad (W_0\,\,&, \,\,W_1) \longmapsto j_*(c_1(h_\alpha)^{n_{\beta}-1}\cap (\psialpha|_E)^*(W_0)) + \psialpha^*(W_1) \label{3.2}
\end{align}
	
where $h_\alpha:= \mathcal{O}_E(1)$ ($E$ thought of as a  $\Pnbeta$-bundle over $ \phialpha^{-1}(\Sigmagamma) $), and $\cap$ denotes the intersection product.\newline
	
\vspace{1ex}
	
Similarly, there exists an isomorphism defined similiarly to $g_\alpha$ above:
	
\begin{align}
\CH_0^{\Q}(\phibeta^{-1}(\Sigmagamma)) &\oplus \CH_1^{\Q}(\Mbeta) \,\,\xrightarrow[\sim]{g_{\beta}} \,\, \CH_1^{\Q}(\mathcal{N}) \label{3.3}\\
\textnormal{given\,\,by}\quad\quad (Z_0\,\,&, \,\,Z_1) \longmapsto j_*(c_1(h_\beta)^{n_{\alpha}-1}\cap (\psibeta|_E)^*(Z_0)) + \psibeta^*(Z_1) \label{3.4}
\end{align}
	
where $h_\beta:= \mathcal{O}_E(1)$ ($E$ thought of as a  $\Pnalpha$-bundle over $ \phibeta^{-1}(\Sigmagamma) $), and $\cap$ denotes the intersection product.
	
\begin{remark}\label{rem-3.5}
	Again, identifying $\Pnalpha \times \Pnbeta$ and $\psialpha^{-1}(\phialpha^{-1}(p))$, it is easy to see that the pull-back bundle $\widetilde{\imath}^*(h_\alpha) \cong \mathcal{O}_{\Pnalpha\times \Pnbeta}(1).$
\end{remark}
	
\vspace{2ex}
	
By lemma \ref{lem-3.3} $\CH_0^{\Q} (\phialpha^{-1}(\Sigmagamma))\cong \Q$, hence the class $(\imath_{\alpha})_*([x])$ will be a $\Q$-basis. 	\newline
	
\begin{lemma}\label{lem-3}
	$g_{\alpha}((\imath_{\alpha})_*([x])) = j_*(\widetilde{\imath}_*{\xtl})$, where $l$ is a line in $\Pnbeta$.
\end{lemma}
	
\begin{proof}
	By \eqref{3.2}, 
	\begin{align}
	g_\alpha((\imath_{\alpha})_*([x])) &= j_*(c_1(h_\alpha)^{n_\beta -1}\cap(\psialpha|_E)^*((\imath_{\alpha})_*([x]))) \nonumber \\
	&= j_*(c_1(h_\alpha)^{n_\beta -1}\cap (\widetilde{\imath}_*[\{x\}\times \Pnbeta])) \hspace{15ex} [\text{Lemma \,\,\ref{lem-3.4}}] \label{3.5}
    \end{align}	
		
		By Remark \ref{rem-3.5} and projection formula applied to $\widetilde{\imath}$ (cf. \cite[Proposition 2.5]{Ful}), 
		\begin{align}
		c_1(h_\alpha)^{n_\beta -1}\cap (\widetilde{\imath}_*[\{x\}\times \Pnbeta]) 
		&= \widetilde{\imath}_* (c_1(\widetilde{\imath}^*(h_{\alpha}))^{n_{\beta}-1} \cap [\{x\} \times \Pnbeta]) \nonumber \\ 
		&= \widetilde{\imath}_*(c_1(\mathcal{O}_{\Pnalpha\times \Pnbeta}(1))^{n_\beta -1}\cap [\{x\}\times \Pnbeta]) \label{3.6}
		\end{align}
		
		But $\mathcal{O}_{\Pnalpha\times \Pnbeta}(1)|_{\{x\}\times \Pnbeta} = \mathcal{O}_{\{x\}\times \Pnbeta}(1)$, which corresponds to the divisor of a hyperplane section $H$ (say), and so by definition of intersection product, 
		\[c_1(\mathcal{O}_{\Pnalpha\times \Pnbeta}(1))\cap [\{x\}\times \Pnbeta] = [\{x\} \times H].\] 
		Repeating this $n_{\beta}-1$ times, we get
		\[c_1(\mathcal{O}_{\Pnalpha\times \Pnbeta}(1))^{n_\beta -1}\cap [\{x\}\times \Pnbeta] = [\{x\} \times l],\]
		where $l$ is a line in $\Pnbeta$. Hence from \eqref{3.5} and \eqref{3.6} we finally get $g_\alpha((\imath_{\alpha})_*([x])) = j_*(\widetilde{\imath}_*{\xtl}),$ as claimed.
	\end{proof}
	
	Let $[x]':= (\imath_{\alpha})_*([x])$ and ${\xtl}' :=\widetilde{\imath}_*{\xtl} $.
	
	
	\begin{proposition}\label{mainprop}
		Let $Z := g_\alpha([x]') \in \CH_1^{\Q}(\mathcal{N}),$ then $Z\neq (\psibeta^*\circ \psibeta_*)(Z).$
	\end{proposition}
	
	\begin{proof}
		Let $j_{\beta}: \phibeta^{-1}(\Sigmagamma)\hookrightarrow \Mbeta$ be the inclusion, so that we have the following blow-up diagram:
		
		\[\xymatrixcolsep{2pc}\xymatrix{ E \ar@{^{(}->}[r]^j \ar[d]_{\substack{\mathbb{P}^{n_\alpha}- \\ bundle}}^{\psibeta|_E} 
			& \mathcal{N} \ar[d]^{\psibeta} \\
			\phibeta^{-1}(\Sigmagamma) \ar@{^{(}->}[r]^(.6){j_{\beta}} 
			& \Mbeta
		}
		\]
		
		If $E = \mathbb{P}(N)$, where $N$ denotes the normal bundle of the embedding $j_{\beta}$, Let $ \mathcal{Q}:= \dfrac{(\psibeta|_E)^*(N)}{\mathcal{O}_E(-1)} $ be the \textit{Excess normal bundle} of rank $n_{\alpha}$, as defined in \cite[\textsection 6.7]{Ful}.
		
		Let $\pi: \mathcal{Q} \rightarrow E$ be the vector bundle projection, and let $ c_{n_{\alpha}}(\mathcal{Q})$ denote the top Chern class of the bundle $\mathcal{Q}$.
		
		By lemma \ref{lem-3}, we want to show that
		\[(\psibeta^{*}\circ \psibeta_*)(j_*\xtl') \neq j_*\xtl'.\]
		
		\textit{To the contrary}, suppose they are equal. We have:
		
		\begin{eqnarray*}
			\textnormal{LHS} &=&(\psibeta^{*} \circ \underbrace{\psibeta_*)(j_*}\xtl') \\
			&=& \underbrace{\psibeta^*(j_{\beta *}} (\psibeta|_E)_* \xtl') \hspace{26ex} [\because \psibeta\circ j = j_{\beta}\circ (\psibeta|_E)] \\
			&=& j_*( c_{n_{\alpha}}(\mathcal{Q}) \cap (\psibeta|_E)^*(\psibeta|_E)_*( \xtl')), \hspace{16ex} [\text{Proposition \,\ref{prop-2.10}}]
		\end{eqnarray*}
		
		So, we would get $ j_*( c_{n_{\alpha}}(\mathcal{Q}) \cap (\psibeta|_E)^*(\psibeta|_E)_*( \xtl')) = j_*\xtl' $.
		
		$\therefore$ denoting $Z:= c_{n_{\alpha}}(\mathcal{Q}) \cap (\psibeta|_E)^*(\psibeta|_E)_*( \xtl') - \xtl'$, we have $j_*(Z) =0$; moreover, $(\psibeta|_E)_*(Z)= 0$ (cf. \cite[Example 3.3.3]{Ful}). \newline
		$\therefore Z=0 $ by \cite[Proposition 6.7(c)]{Ful}, i.e.
		\begin{align}
		\xtl' =  c_{n_{\alpha}}(\mathcal{Q}) \cap (\psibeta|_E)^*(\psibeta|_E)_*( \xtl').
		\end{align}
		
		Moreover, by Lemma \ref{lem-2.8},
		
		\[c_{n_{\alpha}}(\mathcal{Q}) \cap (\psibeta|_E)^*(\psibeta|_E)_*( \xtl') = s^*s_*((\psibeta|_E)^*(\psibeta|_E)_*( \xtl')),
		\]
		
		where $s: E\rightarrow \mathcal{Q}$ denotes the zero section of the bundle map $\mathcal{Q} \xrightarrow{\pi} E$, and $s^*$ is defined as in definition \ref{def-2.8}.
		
		$\therefore$ from (3.7) we would finally get:
		\begin{eqnarray}
		\xtl' = s^*s_* ((\psibeta|_E)^*(\psibeta|_E)_*( \xtl')). \label{3.7}
		\end{eqnarray}
		
		Let us write down the following square:
		\begin{align*}
		\xymatrixcolsep{2pc}\xymatrix{{\Pnalpha \times \Pnbeta} \ar[r]^{\cong} \ar[d]_{p_2} 
			& {\psibeta^{-1}(\phibeta^{-1}(p))} \ar@{^{(}->}[r]^(.6){\widetilde{\imath}} \ar[d]
			& E \ar[d]^{\psibeta|_E}_{\textnormal{flat}} \\
			{\Pnbeta} \ar[r]^{\cong}
			& {\phibeta^{-1}(p)} \ar@{^{(}->}[r]^{\imath}_{\textnormal{proper}}
			& {\phibeta^{-1}(\Sigmagamma)}
		}
		\end{align*}

		Recalling $\xtl' := \widetilde{\imath}_*\xtl$, we get from the diagram above: 
		\begin{eqnarray}
		(\psibeta|_E)^*(\psibeta|_E)_*( \xtl')  &=& (\psibeta|_E)^*\underbrace{(\psibeta|_E)_*(\widetilde{\imath}_*}\xtl) \nonumber \\
		&=& \underbrace{(\psibeta|_E)^*((\imath_*}\circ p_{2*})\xtl) \nonumber \\
		&=& (\widetilde{\imath}_*\circ p_2^*)(p_{2*}\xtl) \nonumber \\
		&=& (\widetilde{\imath}_*\circ p_2^*)([l]) \nonumber \\
		&=& \widetilde{\imath}_*([\Pnalpha\times l])
		\end{eqnarray}
		
		$\therefore$ \eqref{3.7} becomes
		\begin{eqnarray}
		\xtl' &=& s^*\circ s_*(\widetilde{\imath}_*[\Pnalpha \times l]) \label{3.13}\\
		\implies \pi^*(\xtl') &=& s_*( \widetilde{\imath}_*([\Pnalpha \times l])) \hspace{10ex} [\because s^* = (\pi^*)^{-1}]
		\end{eqnarray}
		
		But applying $\pi_* $ to both sides, we see that $\pi_*\pi^*(\xtl') = 0$, since clearly $\pi_*\circ \pi^* =0$ as taking inverse image under a bundle map increases the dimension and then taking image decreases the dimension.\newline
		On the other hand, since $\pi\circ s = \textnormal{Id}_E$, 
		\begin{align*}
		\pi_*\circ s_*(\widetilde{\imath}_*[\Pnalpha \times l]) &= \widetilde{\imath}_*([\Pnalpha \times l]) \hspace{27ex}
		\end{align*}
		
		$\therefore \,\,\textnormal{we would get}\,\,\,\widetilde{\imath}_*([\Pnalpha \times l]) = 0$. But from \eqref{3.13} we would get $[\{x\}\times l]' = 0 $, which would give, by Lemma \ref{lem-3}, that $g_{\alpha}([x]') = 0$, which is a contradiction since $g_\alpha$ is an isomorphism. Hence the claim is proved. 
	\end{proof}
	
\subsection{Proof of the main theorem}
Before stating the main theorem, we make the following  remark:
	
\begin{remark}\label{vectfact}
	Let $V,W$ be two $ \Q $-vector spaces with an isomorphism $\varphi: \Q \langle e \rangle \oplus V \overset{\sim}{\rightarrow} \Q \langle f \rangle \oplus W$. Consider the composite map
	\begin{eqnarray*}
		\Q\langle e \rangle \hookrightarrow \Q \langle e \rangle \oplus V &\overset{\varphi}{\rightarrow}& \Q \langle f \rangle \oplus W  \\
		e &\mapsto& (\psi(e), \phi(e)),
	\end{eqnarray*}
		
and suppose $ \psi(e) \neq 0 $, i.e. the composition below is nonzero:
 \[\Q\langle e \rangle \hookrightarrow \Q \langle e \rangle \oplus V \overset{\varphi}{\rightarrow} \Q \langle f \rangle \oplus W \overset{p_1}{\twoheadrightarrow} \Q \langle f \rangle\] 
Then clearly $\varphi|_V$ induces an isomorphism $V \cong W$, since the composition has to be an isomorphism, as it is a nonzero map between two 1-dimensional spaces, and hence going modulo $\mathbb{Q}\langle e \rangle$ and $\mathbb{Q}\langle f \rangle$ on both sides of $\varphi$, we get our claim.
		
		
		
	
\end{remark}
	
\begin{proposition}\label{prop-3.9}
For generic weights $\alpha, \beta$ in adjacent chambers, the map $g_\beta^{-1}\circ g_\alpha$, when restricted to $\CH_1^{\Q}(\Malpha)$, induces isomorphism 
\[\CH_1^{\Q}(\Malpha)\xrightarrow[\sim]{g_\beta^{-1}\circ g_\alpha} \CH_1^{\Q}(\Mbeta).\]
\end{proposition}
	
\begin{proof}
	Using Lemma \ref{lem-3.3}, let us write $\CH_0^{\Q}(\phialpha^{-1}(\Sigmagamma)) = \Q \langle e \rangle$ and $\CH_0^{\Q}(\phibeta^{-1}(\Sigmagamma)) = \Q \langle f \rangle$, where $e,f$ are some basis elements. Recall the maps $g_\alpha, g_\beta$ from \eqref{3.1} and \eqref{3.3}.\\
	
	Consider the composition
	\begin{eqnarray}
	\mathbb{Q} \langle e \rangle \hookrightarrow \mathbb{Q} \langle e \rangle \oplus \CH_1^{\Q}(\Malpha) \xrightarrow[\sim]{g_{\beta}^{-1}\circ g_{\alpha}} \mathbb{Q} \langle f \rangle \oplus \CH_1^{\Q}(\Mbeta) \overset{p_1}{\twoheadrightarrow}
	\mathbb{Q} \langle f \rangle \label{compositenonzero}
	\end{eqnarray}
	where $p_1$ is the first projection.
	
	According to the remark above, we will be done if we can show that the composition in \eqref{compositenonzero} is nonzero.\\
	Consider the first projection $p_1\circ g_{\beta}^{-1}: CH_1^{\Q}(\mathcal{N})\twoheadrightarrow \Q \langle f \rangle$ with respect to $g_{\beta}$. \\ 
	This map can be described as follows: \\
	We note that the other projection with respect to $g_{\beta}$, namely $p_2\circ g_{\beta}^{-1}:CH_1^{\Q}(\mathcal{N}) \twoheadrightarrow \CH_1^{\Q}(\Mbeta) $ is given by $(\psibeta)_*$, since by \cite[Corollary 9.15]{Voi}
	$ \psibeta_* \circ \psibeta^* = \textnormal{Id}_{\Mbeta} $, and $ \psibeta_* $ sends the terms coming from $ \CH_0^{\Q}(\phibeta^{-1}(\Sigmagamma)) $ to 0, since their image under $\psibeta$ has strictly smaller dimension than the source.\newline
	$\therefore \forall Z \in \CH_1^{\Q}(\mathcal{N})\,\,,\,\, Z = (Z-(\psibeta^* \circ \psibeta_*)(Z)) + \psibeta^*(\psibeta_*(Z))$, and by description of $g_\beta$ in \eqref{3.4}, we get that $Z-(\psibeta^* \circ \psibeta_*)(Z)$ is the first projection with respect to $g_\beta$, i.e.
	\[(p_1\circ g_{\beta}^{-1})(Z) = Z - (\psibeta^* \circ \psibeta_*)(Z).\]
	
	$\therefore$ from Proposition \ref{mainprop} we get that 
	\[(p_1\circ g_{\beta}^{-1})(Z) \neq 0,\,\,\textnormal{where} \,\, Z = g_{\alpha}([x]')\]
	
	also, $[x]'$ is a basis for $CH_0^{\Q}(\phialpha^{-1}(\Sigmagamma)) = \Q\langle e \rangle$; in other words, the composite map in \eqref{compositenonzero} is nonzero. Hence we are done.
	
	
\end{proof}

\begin{theorem}\label{thm-3.10}
	For any two generic weights $\alpha$ and $\beta$, there exists a canonical isomorphism 
	\[\CH_1^{\Q}(\Malpha) \cong \CH_1^{\Q}(\Mbeta).\]
\end{theorem}	

\begin{proof}
By \cite[Lemma 2.7 and Remark 2.9]{BH}, the moduli spaces corresponding to weights in the same chamber are isomorphic. Moreover, we can order the finitely many chambers in such a way that any two consecutive chambers are separated by a single wall. Combining this fact with Proposition \ref{prop-3.9}, we get our claim.	
\end{proof}
\vspace{1ex}

\section{Consequence: description of the Chow group for moduli of parabolic bundles of rank 2 and determinant $\mathcal{O}_X(x)$}

From now on, we consider the case when the rank is 2 and determinant is $\mathcal{O}_X(x)$ for some closed point $x\in X$, and full flags. In rank 2 case, having full flags amounts to giving a 1-dimensional subspace of each fibre over the parabolic points. 
 
\subsection{When the generic weight is small enough}
We recall the following Proposition from \cite{BY}:

\begin{proposition}[BY, Proposition 5.2]\label{prop-4.1}
	Suppose $E$ be a vector bundle of rank $r$ and degree $d$ on $X$. Define the following quantities:
	\begin{align*}
	\epsilon_{\pm}(d,r) &= inf\{\pm(\frac{d}{r} - \frac{d'}{r'}) \,|\, d', r' \in \mathbb{Z}, 1\leq r' < r, and \pm(\frac{d}{r}-\frac{d'}{r'}) > 0\} \\
	\epsilon(d,r) &= min\{\epsilon_{\pm}(d,k) \,|\, k=1,...,r\}
	\end{align*}
	Furthermore, suppose $\sum_{p\in S}\sum_{1=i}^{s_p} m_{p,i}\alpha_{p,i} < \epsilon(d,r)/2$.
	
	\textnormal{(i)} If $E$ is stable as a regular bundle, then $E_*$ is parabolic stable. \\
	\hspace*{2ex}\textnormal{(ii)} If $E_*$ is parabolic stable, then $E$ is semistable as regular bundle.
\end{proposition}

We call a weight $\alpha$ \textit{small} if it satisfies the condition of the Proposition above. 

\begin{lemma}\label{projbund}
	Let us choose a small generic weight $\alpha$ as in Proposition above. There exists a canonical morphism $\Malpha \rightarrow \mathcal{M}$ making $ \Malpha $ into a $(\mathbb{P}^1)^n$-bundle over $\mathcal{M}$, where $n = |S|$.
\end{lemma}

\begin{proof}
Since $\alpha$ is generic, $E_*$ is in fact parabolic stable, so by proposition \ref{prop-4.1} (ii), $E$ is regular semistable bundle (hence stable) as well. Hence there is a map 
\[g: \mathcal{M}_{\alpha} \rightarrow \mathcal{M} \]
 by forgetting the parabolic structure.

For simplicity, first let $n=1$, i.e. only one parabolic point. Recall that $\mathcal{M}$ is a fine moduli space, since $deg(\mathcal{L}) =1.$ (\cite[Proposition 3.2]{BY}). Consider the the \textit{universal (or Poincare)} bundle over $ X\times \mathcal{M} $, whose fibre over each $(p,[E])$ is given by $E_p$. Restrict the bundle over $ \{p\}\times \mathcal{M} $. Call the resulting bundle $ \mathcal{E} $. For each $ [E]\in \mathcal{M} $, the fibre of $\mathbb{P}(\mathcal{E})$ over $ (p,[E]) $ is $ \mathbb{P}(E_p) $, i.e. lines in $ E_p $. Hence the fibre of $\mathbb{P}(\mathcal{E})$ over $(p,[E])$ gives the set of all possible full flags at $E_p$. Moreover, by proposition 4.1 (i), the parabolic bundle $E_*$ resulting from the weight $\alpha$ and parabolic point $p$ will be  automatically parabolic stable. In other words, for each $[E]\in \mathcal{M}$, each point in the fibre of $\mathbb{P}(\mathcal{E})$ over $(p,[E])$ corresponds to a unique point $[E_*] \in \Malpha$. This way we identify $\Malpha$ with $\mathbb{P}(\mathcal{E})$. (cf. \cite[Theorem 4.2]{BY})

In general, if the parabolic data consists of $n$ distinct set of closed points $S=\{p_1,\cdots p_n\}$ and generic weight $\alpha$ as in Proposition \ref{prop-4.1}, for each $i=1,...,n$ let $ \mathcal{E}_i $ denote the restriction of the universal bundle over $ X\times \mathcal{M}$ to $ \{p_i\}\times \mathcal{M} $. Then an analogous argument as above shows that $\mathcal{M}_{\alpha}$ is isomorphic to the fibre product of $\mathbb{P}(\mathcal{E}_i)$'s over $\mathcal{M}$, i.e. \[\mathcal{M}_\alpha \cong \mathbb{P}(\mathcal{E}_1)\times_{\mathcal{M}} \mathbb{P}(\mathcal{E}_2)\times_{\mathcal{M}} ...\times_{\mathcal{M}}\mathbb{P}(\mathcal{E}_n)\].
\end{proof}

\begin{proposition}\label{prop-4.3}
	For small enough generic weights $\alpha$, \[\CH_1^{\Q}(\Malpha) \cong \Q^n\oplus \CH_1^{\Q}(\mathcal{M}), \,\,where\,\, n=|S|.\]
\end{proposition}

\begin{proof}
For each $1\leq i\leq n, n=|S|$, let $ \mathcal{F}_i:= \mathbb{P}(\mathcal{E}_1)\times_{\mathcal{M}} \mathbb{P}(\mathcal{E}_2)\times_{\mathcal{M}} ...\times_{\mathcal{M}} \mathbb{P}(\mathcal{E}_i) $.\newline
By \ref{projbund} we have $\Malpha \cong \mathcal{F}_n$, so we have the following fibre diagram:
\begin{align*}
\xymatrix{{\Malpha} \ar[r] \ar[d] 
	& \mathbb{P}(\mathcal{E}_n) \ar[d] \\
	\mathcal{F}_{n-1} \ar[r] 
	& \mathcal{M}
}
\end{align*}

The left and right vertical arrows above are $\mathbb{P}^1$-bundles, and hence by \cite[Theorem 9.25]{Voi}, there exist isomorphisms of Chow groups:
\begin{align}
\CH_1^{\Q}(\mathbb{P}(\mathcal{E}_n)) &\cong \CH_0^{\Q}(\mathcal{M}) \oplus \CH_1^{\Q}(\mathcal{M}) \nonumber \\
\textnormal{and}\,\,\CH_1^{\Q}(\mathcal{M}_\alpha) \cong \CH_1^{\Q}(\mathcal{F}_n) &\cong \CH_0^{\Q}(\mathcal{F}_{n-1})\oplus \CH_1^{\Q}(\mathcal{F}_{n-1}) \label{4.1}
\end{align} 

Iterating the same for $\mathcal{F}_{n-1}, \mathcal{F}_{n-2}$, and so on, we get from \eqref{4.1}:
\begin{eqnarray*}
	\CH_1^{\Q}(\mathcal{M}_\alpha) \cong \CH_1^{\Q}(\mathcal{F}_n) &\cong& \CH_0^{\Q}(\mathcal{F}_{n-1})\oplus (\CH_0^{\Q}(\mathcal{F}_{n-2})\oplus \CH_1^{\Q}(\mathcal{F}_{n-2})) \\
	&\vdots&  \\
	&\cong& \bigoplus_{i=1}^{n-1}\CH_0^{\Q}(\mathcal{F}_i)\oplus \CH_1^{\Q}(\mathcal{F}_1) \\
	&\cong& \bigoplus_{i=1}^{n-1}\CH_0^{\Q}(\mathcal{F}_i)\oplus \CH_0^{\Q}(\mathcal{M})\oplus \CH_1^{\Q}(\mathcal{M}) \hspace{4ex} \text{\cite[Theorem 9.25]{Voi}}
\end{eqnarray*}

Now, by \cite[Theorem 1.2]{KS} $ \mathcal{M} $ is rational, and hence any projective bundle over it must also be rational; so each $ \mathcal{F}_i $ must be rational. By \cite[Example 16.1.11]{Ful}, the Chow group of 0-cycles is a birational invariant, hence it follows that $ \CH_0^{\Q}(\mathcal{M})\cong \Q $, and $ \CH_0^{\Q}(\mathcal{F}_i)\cong \Q \,\forall i$.

Hence we conclude that 
\begin{equation}
\CH_{1}^{\mathbb{Q}}(\mathcal{M}_\alpha)\cong \mathbb{Q}^n \oplus \CH_{1}^{\mathbb{Q}}(\mathcal{M}).
\end{equation}

\end{proof}

Finally, let us recall the following result due to I. Choe and J. H. Hwang: 
\begin{theorem}[\text{\cite[Main Theorem]{CH}}]\label{thm-4.4}
There is a canonical isomorphism $\CH_1^{\Q}(\mathcal{M}) \cong \CH_0^{\Q}(X)$.
\end{theorem}

We are now able to extend this result for the moduli of parabolic bundles.
\begin{theorem}\label{cor-4.4}
	In case of rank 2 and determinant $\mathcal{O}_X(x)$, for any generic weight $\alpha$, we have \[\CH_1^{\Q}(\mathcal{M}_\alpha) \cong \Q^n \oplus\, \CH_0^{\Q}(X),\,\,\textnormal{where}\,\,n=|S|.\]
\end{theorem}	

\begin{proof}
Combining Proposition \ref{prop-4.3} and Theorem \ref{thm-4.4}, we get $\CH_1^{\Q}(\mathcal{M}_\alpha) \cong \Q^n \oplus\, \CH_0^{\Q}(X)$ for small weights $\alpha$. But using Theorem \ref{thm-3.10}, we can conclude that the same result holds true for arbitrary generic weights as well.
\end{proof}


\begin{thebibliography}{AAAA}
		\bibitem[BH]{BH} H. Boden and Y. Hu, Variations of Moduli of Parabolic Bundles, \textit{Mathematische Annalen}, Volume 301 (1995), 539 -- 559
		\bibitem[BY]{BY} H. Boden and K. Yokogawa, Rationality of the moduli space of parabolic bundles, \textit{Journal of the London Mathematical Society}, Volume 59 (1999), 461 -- 478
		\bibitem[CH]{CH} I. Choe and J. Hwang, Chow group of 1-cycles on the moduli space of vector bundles of rank 2 over a curve, \textit{Mathematische Zeitschrift}, Volume 253 (2006), 253 -- 281
		\bibitem[Ful]{Ful} W. Fulton, Intersection Theory, \textit{Springer-Verlag Berlin Heidelberg} (1998)
		\bibitem[Har]{Har} R. Hartshorne, Algebraic Geometry, \textit{Graduate text in Mathematics, Springer-Verlag}, Volume 52 (1977)
		\bibitem[KS]{KS} A. King and A. Schofield, Rationality of moduli of vector bundles on curves, \textit{Indagationes Mathematicae}
		Volume 10, Issue 4 (1999), 519 -- 535
		\bibitem[MS]{MS} V. B. Mehta and C.S. Seshadri, Moduli of Vector Bundles on Curves with Parabolic Structures, \textit{Mathematische Annalen} Volume 248 (1980), 205 -- 240
		\bibitem[Voi]{Voi} C. Voisin, Hodge Theory and Complex Algebraic Geometry Vol. II (2003), \textit{Cambridge Studies in Advanced Mathematics}
		
	\end{thebibliography}
\end{document}